\documentclass[12pt]{amsart}
\usepackage{amssymb,times}

\makeatletter
\def\@strippedMR{}
\def\@scanforMR#1#2#3\endscan{
  \ifx#1M\ifx#2R\def\@strippedMR{#3}
  \else\def\@strippedMR{#1#2#3}
  \fi\fi}
\renewcommand\MR[1]{\relax\ifhmode\unskip\spacefactor3000 \space\fi
  \@scanforMR#1\endscan
  MR\MRhref{\@strippedMR}{\@strippedMR}}
\makeatother

\parskip=\medskipamount

\newtheorem*{Thm*}{Theorem}
\newtheorem{Thm}{Theorem}
\newtheorem{Cor}[Thm]{Corollary}
\newtheorem{Prop}[Thm]{Proposition}
\newtheorem{Lemma}[Thm]{Lemma}

\theoremstyle{definition}
\newtheorem{Defn}{Definition}
\newtheorem{Notation}[Defn]{Notation}

\newtheorem{Remark}{Remark}

\newcommand{\mf}[1]{\mathbb{#1}}
\newcommand{\mc}[1]{\mathcal{#1}}

\DeclareMathOperator{\NC}{\mathrm{NC}}
\DeclareMathOperator{\Part}{\mathcal{P}}

\DeclareMathOperator{\Inner}{\mathrm{Inn}}
\DeclareMathOperator{\Outer}{\mathrm{Out}}
\DeclareMathOperator{\Sing}{\mathit{Sing}}

\newcommand{\norm}[1]{\left\Vert#1\right\Vert}
\newcommand{\abs}[1]{\left\vert#1\right\vert}
\newcommand{\chf}[1]{\mathbf{1}_{#1}}
\newcommand{\set}[1]{\left\{#1\right\}}
\newcommand{\ip}[2]{\left \langle #1, #2 \right \rangle}

\renewcommand{\phi}{\varphi}

\allowdisplaybreaks[1]

\title{Two-state free Brownian motions}
\author{Michael Anshelevich}
\thanks{This work was supported in part by NSF grant DMS-0900935}
\address{Department of Mathematics, Texas A\&M University, College Station, TX 77843-3368}
\email{manshel@math.tamu.edu}
\date{\today}

\begin{document}

\begin{abstract}
In a two-state free probability space $(\mc{A}, \phi, \psi)$, we define an \emph{algebraic two-state free Brownian motion} to be a process with two-state freely independent increments whose two-state free cumulant generating function $R^{\phi, \psi}(z)$ is quadratic. Note that a priori, the distribution of the process with respect to the second state $\psi$ is arbitrary. We show, however, that if $\mc{A}$ is a von Neumann algebra, the states $\phi, \psi$ are normal, and $\phi$ is faithful, then there is only a one-parameter family of such processes. Moreover, with the exception of the actual free Brownian motion (corresponding to $\phi = \psi$), these processes only exist for finite time.
\end{abstract}


\maketitle

\section{Introduction}

The study of free probability was initiated by Voiculescu in the early 1980s \cite{Voi85}. While free probability has crucial applications to the study of operator algebras and random matrices, it has also developed into a deep and sophisticated theory in its own right. As one illustration, consider the free Central Limit Theorem. Its formulation is the same as for the usual CLT, with two changes. First, the objects involved are noncommutative random variables, that is, elements of a noncommutative $\ast$-algebra (or $C^\ast$-algebra, or von Neumann algebra $\mc{A}$, or the algebra of operators affiliated to it), with a state $\phi$ which replaces the expectation functional. Second, independence is replaced by Voiculescu's free independence, which is more appropriate for non-commuting objects. The algebraic version of the theorem was proved in \cite{Voi85}, followed by the full analytic version for identically distributed triangular arrays in \cite{Pata-CLT} and general triangular arrays in \cite{Chistyakov-Gotze-Limit-I}. Note that in the analytic theorems, the hypothesis on the distributions are identical to those in the usual CLT. On the other hand, in \cite{BerVoiSuperconvergence} and \cite{Wang-Local-CLT} the authors showed that the mode of convergence in the free CLT is actually much stronger than the classical convergence in distribution. In all these results, the limiting distribution is the semicircle law. It is characterized by having zero free cumulants of order greater than $2$ (of course, in most of these results, no a priori assumption on the existence of free cumulants is made).

An important point about free probability is that, as mentioned above, there are different settings in which the theory can be studied. Consider the notion of (reduced) free product, related to the discussion above by the property that different components in a free product are freely independent. One can take a reduced free product of $\ast$-algebras with states, or $C^\ast$-algebras with representations, or of Hilbert spaces, or of von Neumann algebras with states, and all of these constructions are consistent. One can frequently extend purely algebraic results to the more analytic context of normed algebras (although, as illustrated in \cite{BV93}, such extensions are often non-trivial).

This paper is about a related theory where this is no longer the case. In \cite{Boz-Spe-Psi-independent,BLS96}, Bo{\.z}ejko, Leinert, and Speicher constructed what they called a conditionally free probability theory, which we will refer to as two-state free probability theory. The setting is now a $\ast$-algebra (von Neumann algebra, etc.) $\mc{A}$ with \emph{two} states, say $\phi$ and $\psi$. Initially the authors had a single example of such a structure, but the theory has since been quite successful, at least in two settings. For results concerning single distributions, for example the study of limit theorems, see \cite{Krystek-Conditional,Belinschi-C-free-unbounded,Boz-Bryc-Two-states,Wang-Additive-c-free}; on the other hand, for results in the purely algebraic setting, see \cite{Popa-c-free-amalgamation} and \cite{AnsAppell3}. However, very little work on this theory has been done in the analytic setting; in fact, we are only aware of one article \cite{Ricard-t-Gaussian}. In this paper we show that this is not a coincidence, by the following example. We define what is natural to call (algebraic) two-state free Brownian motions. This is a very large class of processes, since the ``Brownian motion'' property only determines the relative position of the expectations $\phi$ and $\psi$, but the choice of $\psi$ is arbitrary, at least in the algebraic setting. We then show that if $\mc{A}$ is a von Neumann algebra, and the expectation $\phi$ is faithful, then out of this infinite-dimensional family only a one-parameter family of processes can actually be realized. Moreover, with the exception of the actual free Brownian motion (corresponding to the case $\phi = \psi)$, these processes only exist on a finite time interval.

The paper is organized as follows. After the introduction and a background section, in Section~\ref{Section:Uniqueness} we define the two-state free Brownian motions, and show that if $\phi$ is faithful, only a one-parameter family of these processes may exist. The method of proof involves stochastic integration. In Section~\ref{Section:Existence} we show that this one-parameter family actually does exist, by using a Fock space construction. We show that these processes are not Markov, even though they have classical versions, the time-reversed free Poisson processes of \cite{Bryc-Wesolowski-Bi-Poisson}. We also compute the generators of these processes. Finally, Section~\ref{Section:Cstar} contains some comments on the case when $\mc{A}$ is a $C^\ast$- rather than a von Neumann algebra. In particular, in this section we give another characterization of the one-parameter family mentioned above: in a large class, these are the only processes whose higher variation processes converge to the appropriate limits in $L^\infty(\phi)$ rather than just in $L^2(\phi)$.

\textbf{Acknowledgements.} I am grateful to Ken Dykema and Laura Matusevich for important comments. Thanks also to W{\l}odek Bryc for giving me an early version of the article \cite{Bryc-Markov-Meixner}, which led (in a very indirect way) to the present work.

\section{Preliminaries}

\subsection{Partitions.}

$\Part(n)$ will denote the lattice of all partitions of a set of $n$ elements (into nonempty, pairwise disjoint subsets called blocks). The number of blocks of $\pi$ is denoted $\abs{\pi}$. The partitions are ordered by reverse refinement, so that $\hat{0} = \set{(1), (2), \ldots, (n)}$ is the smallest and $\hat{1} = \set{(1, 2, \ldots, n)}$ is the largest partition. In the lattice, $\sigma \vee \pi$ is the smallest partition which is larger than both $\sigma$ and $\pi$, and $\sigma \wedge \pi$ is the largest partition which is smaller than both $\sigma$ and $\pi$.

$\NC(n)$ is the sub-lattice of non-crossing partitions, which have the property that whenever $x_1 < y_1 < x_2 < y_2$ with $x_1, x_2 \in U$ and $y_1 , y_2 \in V$, where $U, V$ are blocks of the partition $\pi$, then $U = V$. In a non-crossing partition $\pi$, a block $V$ is inner if for some $y_1, y_2 \not \in V$ and all $x \in V$, $y_1 < x < y_2$, otherwise it is called outer. Denote by $\Inner(\pi)$ all the inner blocks of the non-crossing partition $\pi$, and by $\Outer(\pi)$ the outer blocks. Also, denote by $\NC_{1,2}(n)$ all the non-crossing partitions into pairs and singletons, in other words partitions with all $\abs{V} \leq 2$. $\Sing(\pi)$ are all the singleton blocks of a partition.

If $f$ is a function of $k < n$ arguments and $V \subset \set {1, \ldots, n}$, $V = \set{i(1) < i(2) < \ldots < i(k)}$, then we denote
\[
f(x_1, x_2, \ldots, x_n : V) = f(x_{i(1)}, x_{i(2)}, \ldots, x_{i(k)}).
\]

\subsection{Jacobi parameters.}

If $\nu$ is a probability measure on $\mf{R}$ all of whose moments are finite, it has associated to it two sequences of Jacobi parameters,
\begin{equation*}
J(\nu) =
\begin{pmatrix}
\beta_0, & \beta_1, & \beta_2, & \beta_3, & \ldots \\
\gamma_1, & \gamma_2, & \gamma_3, & \gamma_4, & \ldots
\end{pmatrix}.
\end{equation*}
There are numerous ways of defining these parameters, using orthogonal polynomials, tridiagonal matrices, or Viennot-Flajolet theory. For our purposes, the most convenient definition is the following. The Cauchy transform of $\nu$ can be expanded into a formal power series
\[
G_\nu(z) = \int_{\mf{R}} \frac{1}{z - x} \,d\nu(x) = \sum_{n=0}^\infty m_n(\nu) \frac{1}{z^{n+1}},
\]
where $m_n(\nu) = \int_{\mf{R}} x^n \,d\nu(x)$ is the $n$-th moment of $\nu$. Then we also have a continued fraction expansion
\[
G_\nu(z) =
\cfrac{1}{z - \beta_0 -
\cfrac{\gamma_1}{z - \beta_1 -
\cfrac{\gamma_2}{z - \beta_2 -
\cfrac{\gamma_3}{z - \beta_3 -
\cfrac{\gamma_4}{z - \ldots}}}}}.
\]
If some $\gamma = 0$, the continued fraction terminates, in which case the subsequent $\beta$ and $\gamma$ coefficients can be defined arbitrarily. See \cite{Chihara-book} for more details.

The monic orthogonal polynomials $\set{P_n}$ for $\nu$ satisfy a recursion relation
\[
x P_n(x) = P_{n+1} + \beta_n P_n(x) + \gamma_n P_{n-1}(x),
\]
with $P_{-1} (x) = 0$.

Finally, we define the map $\Phi_t$ (Jacobi shift, an inverse of coefficient stripping) on probability measures with finite moments by
\begin{equation*}
J(\Phi_t[\nu]) =
\begin{pmatrix}
0, & \beta_0, & \beta_1, & \beta_2, & \beta_3, & \ldots \\
t, & \gamma_1, & \gamma_2, & \gamma_3, & \gamma_4, & \ldots
\end{pmatrix}.
\end{equation*}
for $\nu$ as above. Equivalently,
\[
G_{\Phi_t[\nu]}(z) = \frac{1}{z - t G_\nu(z)},
\]
and this last definition makes sense for general probability measures, see Remark~4.3 of \cite{Belinschi-Nica-B_t}.

\subsection{Two-state free probability theory.}

In this section, $(\mc{A}, \phi, \psi)$ is an algebraic two-state non-commutative probability space, that is, $\mc{A}$ is a star-algebra and $\phi, \psi$ are positive unital functionals on it. $\mc{A}^{sa}$ will denote the self-adjoint part of $\mc{A}$.

For $X_1, X_2, \ldots, X_n \in \mc{A}$, define the free cumulant functionals $R^\psi(X_1, X_2, \ldots, X_n)$ via
\[
\psi \left[ X_1 X_2 \ldots X_n \right] = \sum_{\pi \in \NC(n)} \prod_{V \in \pi} R^\psi(X_1, X_2, \ldots, X_n : V)
\]
and the two-state free cumulant functionals $R^{\phi, \psi}(X_1, X_2, \ldots, X_n)$ via
\[
\phi \left[ X_1 X_2 \ldots X_n \right] = \sum_{\pi \in \NC(n)} \prod_{V \in \Outer(\pi)} R^{\phi, \psi}(X_1, X_2, \ldots, X_n : V) \prod_{V \in \Inner(\pi)} R^\psi(X_1, X_2, \ldots, X_n : V).
\]
Then both $R^\psi$ and $R^{\phi, \psi}$ are multilinear functionals.

Denote $R^{\phi, \psi}_n(X) = R^{\phi, \psi}(X, X, \ldots, X)$, where $X$ is repeated $n$ times.

If $X$ has distribution $\mu$ with respect to $\phi$ and $\nu$ with respect to $\psi$, its two-state free cumulant generating function is
\[
R^{\mu, \nu}(z) = R^{\phi, \psi}_X(z) = \sum_{n=1}^\infty R^{\phi, \psi}_n(X) z^n.
\]
The definition of the free cumulant generating function $R^\nu(z) = R^\psi_X(z)$ is similar.

\begin{Defn}
Let $\mc{A}_1, \mc{A}_2, \ldots, \mc{A}_k \subset \mc{A}$ be a family of subalgebras.
\begin{enumerate}
\item
This family is $\psi$-freely independent if for any $a_1, a_2, \ldots, a_n \in \mc{A}$,
\[
a_j \in \mc{A}_{i(j)}, \quad i(j) \neq i(j+1)
\]
with all $\psi[a_j] = 0$, we have
\[
\psi[a_1 a_2 \ldots a_n] = 0.
\]
\item
This family is two-state freely independent if it is $\psi$-freely independent and, under the same assumptions on $a_1, \ldots, a_n$, also
\[
\phi[a_1 a_2 \ldots a_n] = \phi[a_1] \phi[a_2] \ldots \phi[a_n].
\]
\end{enumerate}
\end{Defn}

\begin{Thm*}[Theorem 3.1 of \cite{BLS96}]
Let $\mc{A}_1, \mc{A}_2, \ldots, \mc{A}_k \subset \mc{A}$ be a family of subalgebras.
\begin{enumerate}
\item
This family is $\psi$-freely independent if and only if for any
\[
a_1, a_2, \ldots, a_n \in \cup_{i=1}^k \mc{A}_i,
\]
we have
\[
R^\psi(a_1, a_2, \ldots, a_n) = 0
\]
unless all $a_1, a_2, \ldots, a_n \in \mc{A}_j$ for the same $j$.
\item
This family is two-state freely independent if and only if it is $\psi$-freely independent and also
\[
R^{\phi, \psi}(a_1, a_2, \ldots, a_n) = 0
\]
unless all $a_1, a_2, \ldots, a_n \in \mc{A}_j$ for the same $j$.
\end{enumerate}
\end{Thm*}

\section{Almost uniqueness of the two-state free Brownian motion}
\label{Section:Uniqueness}

\subsection{Algebraic framework}

\begin{Defn}
A family $\set{X(t), 0 \leq t \leq T} \subset \mc{A}^{sa}$ is a \emph{process with two-state freely independent increments} if the increments
\[
X([s,t)) = X(t) - X(s)
\]
of this process corresponding to disjoint intervals are two-state freely independent. For convenience, we will also assume that $X(0) = 0$, $\phi[X(t)] = 0$ for all $t$, and the distributions of the increments of the process with respect to both $\phi$ and $\psi$ are stationary. These assumptions can be dropped.

From now on, whenever we are considering such a process, we will assume that $\set{X(t): 0 \leq t \leq T}$ generate $\mc{A}$.
\end{Defn}

\begin{Defn}
$X$ has a \emph{two-state normal distribution} if $R^{\phi, \psi}_n(X) = 0$ for $n > 2$.
\end{Defn}

\begin{Remark}
The justification for this definition is that such random variables appear in the two-state free central limit theorem, see Theorem 4.3 in \cite{BLS96} and Lemma~7 and Remark~3 in \cite{AnsAppell3}. Theorem~\ref{Thm:Expectations} show that alternatively, we could define a two-state free Brownian motion below by requiring it to have zero higher variations, computed as limits in $L^2(\phi)$.
\end{Remark}

\begin{Defn}
A family $\set{X(t), 0 \leq t \leq T}$ is an \emph{algebraic two-state free Brownian motion} if it is a process with two-state freely independent increments which have two-state normal distributions, $\phi$-mean zero, and $\phi$-variance
\[
\phi[(X(t) - X(s))^2] = (t-s).
\]
\end{Defn}

\subsection{Analytic framework}

Throughout the rest of the paper, $(\mc{A}, \phi, \psi)$ will be a $W^\ast$-non\-com\-mu\-ta\-tive probability space, that is, $\mc{A}$ a von Neumann algebra, $\phi$ a faithful normal state on $\mc{A}$, and $\psi$ a normal (typically not faithful) state on $\mc{A}$.

We will call algebraic two-state free Brownian motions that exist in this setting simply two-state free Brownian motions.

We say that $A_n \rightarrow A$ in $L^2(\phi)$ if $\phi [\abs{A_n - A}^2] \rightarrow 0$.

\begin{Notation}
In a number of proofs, we will fix a time $T > 0$. In that case, we denote $X = X(T)$ and
\[
X_{i, N} = X \left(\frac{i}{N} T \right) - X \left(\frac{i-1}{N} T \right).
\]
Note that the two-state free independence of increments implies
\[
R^{\psi}_k (X_{i, N}) = \frac{1}{N} R^{\psi}_k(X), \qquad R^{\phi, \psi}_k(X_{i, N}) = \frac{1}{N} R^{\phi, \psi}_k(X).
\]
\end{Notation}

\begin{Thm}
\label{Thm:Expectations}
For a process with two-state freely independent increments,
\[
\lim_{N \rightarrow \infty} \phi \left[ \left(\sum_{i=1}^N X_{i,N}^k \right)^2 \right] = R^{\phi, \psi}_{k}(X)^2 + R^{\phi, \psi}_{2k}(X)
\]
and
\[
\lim_{N \rightarrow \infty} \phi \left[ \left(\sum_{i=1}^N X_{i,N}^2 - R^{\phi, \psi}_2(X) \right)^2 \right] = R^{\phi, \psi}_4(X).
\]
In particular, for a two-state free Brownian motion, in $L^2(\phi)$
\[
\lim_{N \rightarrow \infty} \sum_{i=1}^N X_{i,N}^2 = R^{\phi, \psi}_2(X)
\]
and for $k > 2$,
\[
\lim_{N \rightarrow \infty} \sum_{i=1}^N X_{i,N}^k = 0.
\]
We record these results symbolically as
\[
\int_0^T (dX(t))^2 = T
\]
and $\int_0^T (dX(t))^k = 0$ for $k > 2$.
\end{Thm}

\begin{proof}
\[
\begin{split}
\phi \left[ \left(\sum_{i=1}^N X_{i,N}^k \right)^2 \right]
&= \sum_{i \neq j}^N \phi \left[ X_{i,N}^k X_{j,N}^k \right]+ \sum_{i=1}^N \phi \left[ X_{i,N}^{2k} \right] \\
& = N (N-1) \phi[X_{1,N}^k]^2 + N \phi[X_{1,N}^{2k}] \\
\end{split}
\]
Now using free cumulant expansions, this expression equals
\[
\begin{split}
&\quad = N (N-1) \left( \frac{1}{N} R^{\phi, \psi}_{k}(X) + \sum_{\substack{\pi \in \NC(k) \\ \pi \neq \hat{1}}} \frac{1}{N^{\abs{\pi}}} \prod_{V \in \Outer(\pi)} R^{\phi, \psi}_{\abs{V}}(X) \prod_{V \in \Inner(\pi)} R^\psi_{\abs{V}}(X) \right)^2 \\
&\qquad\qquad + N \left( \frac{1}{N} R^{\phi, \psi}_{2k}(X) + \sum_{\substack{\pi \in \NC(2k) \\ \pi \neq \hat{1}}} \frac{1}{N^{\abs{\pi}}} \prod_{V \in \Outer(\pi)} R^{\phi, \psi}_{\abs{V}}(X) \prod_{V \in \Inner(\pi)} R^\psi_{\abs{V}}(X) \right) \\
&\quad \stackrel{N \rightarrow \infty}{\longrightarrow} R^{\phi, \psi}_{k}(X)^2 + R^{\phi, \psi}_{2k}(X).
\end{split}
\]
Also,
\[
\begin{split}
& \phi \left[ \left(\sum_{i=1}^N X_{i,N}^2 - R^{\phi, \psi}_2(X) \right)^2 \right] \\
&\quad = \sum_{i \neq j}^N \phi \left[ X_{i,N}^2 X_{j,N}^2 \right]+ \sum_{i=1}^N \phi \left[ X_{i,N}^{4} \right] - 2 R^{\phi, \psi}_2(X) \sum_{i=1}^N \phi \left[ X_{i,N}^{2} \right] + R^{\phi, \psi}_2(X)^2 \\
&\quad \stackrel{N \rightarrow \infty}{\longrightarrow} R^{\phi, \psi}_2(X)^2 + R^{\phi, \psi}_4(X) - 2 R^{\phi, \psi}_2(X)^2 + R^{\phi, \psi}_2(X)^2  = R^{\phi, \psi}_4(X). \qedhere
\end{split}
\]
\end{proof}

\begin{Prop}
\[
\lim_{N \rightarrow \infty} \psi \left[ \sum_{i=1}^N X_{i,N}^k \right] = R^\psi_k(X).
\]
We record this result symbolically as
\[
\int_0^T \psi \left[ (dX(t))^k \right] = R^\psi_k(X(T)).
\]
\end{Prop}

The proof is similar to the preceding theorem, using the $\psi$-free independence of the increments of the process $\set{X(t)}$.

\begin{Cor}
\label{Cor:Variations}
For a two-state free Brownian motion, $R^\psi_k(X) = 0$ for $k > 2$ and $R^{\psi}_2(X(T)) = R^{\phi, \psi}_2(X(T))$. If $R^\psi_1(X(T)) = \alpha T$, we call the corresponding process a \emph{two-state free Brownian motion with parameter $\alpha$}.
\end{Cor}

\begin{Cor}
\label{Corollary:Densities}
Denote by $\mu_t$ the distribution of $X(t)$ with respect to $\phi$, and by $\nu_t$ the corresponding distribution with respect to $\psi$. Then
\begin{equation}
\label{nu}
d\nu_t(x) = \frac{1}{2 \pi t} \sqrt{\Bigl( 4 t - (x - \alpha t)^2 \Bigr)_+} \,dx
\end{equation}
and for $\alpha \neq 0$,
\begin{equation}
\label{mu}
d\mu_t(x) = \frac{1}{2 \pi t} \frac{\sqrt{\Bigl( 4 t - (x - \alpha t)^2 \Bigr)_+}}{1 + \alpha x} \,dx + \max \left(1 - \frac{1}{\alpha^2 t}, 0 \right) \delta_{-1/\alpha}.
\end{equation}
Also denote $C_T = 1 + \alpha X(T)$. Then the distribution of $C_T$ with respect to $\phi$ is
\begin{equation}
\label{C_T-density}
\frac{1}{2 \pi} \frac{1}{\alpha^2 T} \frac{\sqrt{\Bigl(((1 + \alpha \sqrt{T})^2 - y) (y - (1 - \alpha \sqrt{T})^2) \Bigr)_+}}{y} \,dy + \max \left(1 - \frac{1}{\alpha^2 T}, 0 \right) \delta_0.
\end{equation}
\end{Cor}

\begin{proof}
Since $\set{X(t)}$ is a two-state free Brownian motion, the two-state free cumulant generating function is
\begin{equation}
\label{R-mu-nu}
R^{\mu_t, \nu_t}(z) = t z^2.
\end{equation}
Moreover, the preceding corollary implies that the free cumulant generating function of $\nu_t$ is
\[
R^{\nu_t}(z) = \alpha t z + t z^2
\]
for some $\alpha$. Therefore $\nu_t$ is a semicircular distribution with Jacobi parameters
\[
J(\nu_t) =
\begin{pmatrix}
\alpha t, & \alpha t, & \alpha t, & \alpha t, & \ldots \\
t, & t, & t, & t, & \ldots
\end{pmatrix}
\]
and formula~\eqref{nu} holds. If $\alpha = 0$, then $\mu_t = \nu_t$ and the process is the free Brownian motion, so throughout the rest of the paper we will assume that $\alpha \neq 0$. By Lemma~7 of \cite{AnsAppell3}, equation~\eqref{R-mu-nu} implies that $\mu_t = \Phi_t[\nu_t]$, so $\mu_t$ is a free Poisson distribution with Jacobi parameters
\begin{equation}
\label{Jacobi-mu}
J(\mu_t) =
\begin{pmatrix}
0, & \alpha t, & \alpha t, & \alpha t, & \ldots \\
t, & t, & t, & t, & \ldots
\end{pmatrix}.
\end{equation}
This implies the density formula~\eqref{mu}. Note that even though these distributions are free Poisson, $t$ is not the free convolution parameter. The last formula follows via the substitution $y = 1 + \alpha x$.
\end{proof}

\begin{Remark}
\label{Remark:Bi-Poisson}
The classical version of the process $\set{X(t)}$ is a particular case of the free Bi-Poisson process from \cite{Bryc-Wesolowski-Bi-Poisson}, corresponding to $\eta = \alpha$ and the rest of the parameters equal zero. In fact, this process is a time-reversed free Poisson process: if $Y(t) = t X(1/t)$, then the $\psi$-distribution of $Y(t)$ has Jacobi parameters
\[
\begin{pmatrix}
\alpha, & \alpha, & \alpha, & \alpha, & \ldots \\
t, & t, & t, & t, & \ldots
\end{pmatrix}
\]
and so is semicircular with mean $\alpha$ and variance $t$, and its $\phi$-distribution $\tau_t$ has Jacobi parameters
\begin{equation*}
J(\tau_t) =
\begin{pmatrix}
0, & \alpha, & \alpha, & \alpha, & \ldots \\
t, & t, & t, & t, & \ldots
\end{pmatrix}.
\end{equation*}
So $\tau_t$ is the centered free Poisson distribution with parameter $\alpha$, and $\set{\tau_t}$ form a free (rather than a two-state free) convolution semigroup.
\end{Remark}

\begin{Prop}
The process $\set{Y(t)}$ itself is \emph{not} the free Poisson process.
\end{Prop}

\begin{proof}
We compute, for $s < t$,
\[
Y(t) - Y(s) = t X(1/t) - s X(1/s)
= (t-s) X(1/t) - s X([1/t, 1/s)).
\]
The free cumulant generating function of the $\psi$-distribution of $(t-s) X(1/t)$ is
\[
R^\psi_{(t-s) X(1/t)}(z) = \alpha \frac{t-s}{t} z + \frac{(t-s)^2}{t} z^2
= \alpha \left( 1 - \frac{s}{t} \right) z + \left(t - 2 s + \frac{s^2}{t} \right) z^2.
\]
Similarly, the free cumulant generating function of the $\psi$-distribution of $- s X([1/t, 1/s))$ is the same as that of $-s X(1/s - 1/t)$, in other words
\[
R^\psi_{- s X([1/t, 1/s))}(z) = \alpha (-s) \left( \frac{1}{s} - \frac{1}{t} \right) z + s^2 \left( \frac{1}{s} - \frac{1}{t} \right) z^2
= \alpha \left( -1 + \frac{s}{t} \right) z + \left( s - \frac{s^2}{t} \right) z^2.
\]
Since $(t-s) X(1/t)$ and $- s X([1/t, 1/s))$ are freely independent with respect to $\psi$, it follows that the free cumulant generating function of $Y(t) - Y(s)$ is the sum
\[
R^\psi_{(t-s) X(1/t)}(z) + R^\psi_{- s X([1/t, 1/s))}(z) = (t-s) z^2.
\]

Similarly,
\[
R^{\phi, \psi}_{(t-s) X(1/t)}(z) = \frac{(t-s)^2}{t} z^2
= \left(t - 2 s + \frac{s^2}{t} \right) z^2,
\]
\[
R^{\phi, \psi}_{- s X([1/t, 1/s))}(z) = s^2 \left( \frac{1}{s} - \frac{1}{t} \right) z^2
= \left( s - \frac{s^2}{t} \right) z^2,
\]
and
\[
R^{\phi, \psi}_{(t-s) X(1/t)}(z) + R^{\phi, \psi}_{- s X([1/t, 1/s))}(z) = (t-s) z^2.
\]
We conclude that $Y(t) - Y(s)$ has, with respect to $\phi$, the centered semicircular distribution with variance $t-s$, which is clearly different from the distribution of $Y(t-s)$.
\end{proof}

\begin{Prop}
Let $\set{X(t): 0 \leq t \leq S}$, $S > T$ be a process with two-state freely independent increments and $A \in \mc{A}^{sa}$ two-state free from it. Assume that $\psi[A] = 0$. Then in $L^2(\phi)$,
\[
\lim_{N \rightarrow \infty} \sum_{i=1}^N  X_{i,N} \, A \, X_{i,N} = 0.
\]
Symbolically, $\int_0^T  dX(t) \, A \,dX(t) = 0$.
\end{Prop}

\begin{proof}
\[
\phi \left[ \left(\sum_{i=1}^N  X_{i,N} \, A \, X_{i,N} \right)^2 \right]
= \sum_{i \neq j}^N \phi[ X_{i,N} A X_{i,N} X_{j,N} A X_{j,N}] + \sum_{i=1}^N \phi[ X_{i,N} A X_{i,N}^2 A X_{i,N}]
\]
For the first term, since $X_{i,N}$ are $\phi$-centered, $A$ is $\psi$-centered, and they are two-state freely independent among themselves, a cumulant expansion shows that each term of the sum is zero. For the second term,
\[
\phi[ X_{i,N} A X_{i,N}^2 A X_{i,N}]
= \phi[X_{i,N}^2]^2 \psi[A^2]
= \frac{1}{N^2} R^{\phi, \psi}(X)^2 \psi[A^2].
\]
So as $N \rightarrow \infty$, both terms above converge to zero.
\end{proof}

\begin{Cor}
If $A$ is two-state free from $\set{X(t): S_1 \leq t \leq S_2}$, then
\[
\int_{S_1}^{S_2} dX(t) \, A \, dX(t) = \psi[A] \int_{S_1}^{S_2} (dX(t))^2.
\]
In particular, for a two-state free Brownian motion,
\[
\psi[A] = \frac{1}{S_2 - S_1} \int_{S_1}^{S_2} dX(t) \, A \, dX(t).
\]
\end{Cor}

\begin{Cor}
Let $\set{X(t): 0 \leq t \leq T}$ be a process with two-state freely independent increments in $(\mc{A}, \phi, \psi)$, which generates $\mc{A}$. Then $\psi$ is uniquely determined by $\phi$ and the process.
\end{Cor}

\begin{proof}
For each $S$, for $A \in W^\ast(\set{X(t): t < S})$ we have
\[
\int_S^T dX(t) \, A \, dX(t) = \psi[A] \int_S^T (dX(t))^2,
\]
the integrals being defined in $L^2(\phi)$. But
\[
\bigcup_{0 \leq S < T} W^\ast(\set{X(t): t < S})
\]
is dense in $\mc{A}$.
\end{proof}

\section{Existence proof}
\label{Section:Existence}

While the results in this section do not follow directly from \cite{Ricard-t-Gaussian}, numerous arguments in it are adapted from that article.

If $\alpha = 0$, the process is the free Brownian motion. So in the following discussion, we assume that $\alpha \neq 0$.

\begin{Remark}[Fock space construction]
For $\mc{H} = L^2(\mf{R}_+, \,dx)$, denote by
\[
\mc{F}(\mc{H}) = \bigoplus_{n=0}^\infty \mc{H}_{\mf{C}}^{\otimes n} = \mf{C} \Omega \oplus \mc{H}_{\mf{C}} \oplus \mc{H}_{\mf{C}}^{\otimes 2} \oplus \ldots
\]
its full Fock space, with the standard inner product. For each $g \in L^1(\mf{R}_+, \,dx) \cap L^2(\mf{R}_+, \,dx)$, define on $\mc{F}(\mc{H})$ operators
\begin{equation}
\label{Fock-rep}
\begin{split}
X(g) \Omega & = g, \\
X(g) (f_1 \otimes f_2 \otimes \ldots \otimes f_n)
& = g \otimes f_1 \otimes f_2 \otimes \ldots \otimes f_n \\
&\quad + \alpha \langle g \rangle f_1 \otimes f_2 \otimes \ldots \otimes f_n \\
&\quad + \ip{g}{f_1} f_2 \otimes \ldots \otimes f_n,
\end{split}
\end{equation}
where $\langle g \rangle = \int_{\mf{R}_+} g(x) \,dx$. In other words, on any tensor of length at least one, $X(g)$ acts in the same way as $a^+(g) + a^-(g) + \alpha \langle g \rangle$, where $a^+, a^-$ are the free creation and annihilation operators. Denote $X(I) = X(\chf{I})$ for an interval $I$, and $X(t) = X([0,t))$.

Denote $\mc{A} = W^\ast(\set{X(t): t \geq 0})$ and on $\mc{A}$, define the state $\phi$ by
\[
\phi[A] = \ip{\Omega}{A \Omega}.
\]
Denote $C_T = 1 + \alpha X(T)$, and define the functional $\psi_T$ by
\[
\psi_T[A] = \phi[A C_T].
\]
Note that $\psi_T[1] = \phi[C_T] = 1$. We will show that $\psi_T$ is positive, and so a state.
\end{Remark}

\begin{Lemma}
\label{Lemma:OP}
Let $\set{P_n(x,t)}$ be the monic orthogonal polynomials with respect to $\nu_t$, and $\set{Q_n(x,t)}$ the monic orthogonal polynomials with respect to $\mu_t$. Then
\begin{enumerate}
\item
$P_n(x,t) = U_n(x - \alpha t, t)$, where $U_n$ are the Chebyshev polynomials of the second kind.
\item
$Q_n(x,t) = P_n(x,t) + \alpha t P_{n-1}(x,t)$ and
\[
(1 + \alpha x) P_n(x,t), = \alpha Q_{n+1}(x,t) + Q_n(x,t).
\]
\item
For any $n \geq 0$ and any $t$,
\[
\int_{\mf{R}^n} x^n \,d\nu_t(x) = \int_{\mf{R}} x^n (1 + \alpha x) \,d\mu_t(x).
\]
\end{enumerate}
\end{Lemma}

\begin{proof}
The proofs of parts (a) and (b) follow from the recursion relations
\[
x P_n(x,t) = P_{n+1}(x,t) + \alpha t P_n(x,t) + t P_{n-1}(x,t)
\]
and $x Q_0(x,t) = Q_1(x,t)$,
\[
x Q_n(x,t) = Q_{n+1}(x,t) + \alpha t Q_n(x,t) + t Q_{n-1}(x,t).
\]
The proof of part (c) follows from the density formulas of Corollary~\ref{Corollary:Densities}; note that the statement remains true even if $\mu_t$ has an atom.
\end{proof}

\begin{Lemma}
\label{Lemma:Fock-product}
Let $I_1, I_2, \ldots, I_n \subset [0,T]$ be intervals such that all $I_i \cap I_{i+1} = \emptyset$.
Then
\[
P_{k(1)}\bigl(X(I_1), \abs{I_1}\bigr) \ldots P_{k(n-1)}\bigl(X(I_{n-1}), \abs{I_{n-1}}\bigr) Q_{k(n)}\bigl(X(I_n), \abs{I_n}\bigr) \Omega
= \chf{I_1}^{\otimes k(1)} \otimes \ldots \otimes \chf{I_n}^{\otimes k(n)}.
\]
In particular,
\[
\phi \Bigl[ P_{k(1)}\bigl(X(I_1), \abs{I_1}\bigr) \ldots P_{k(n-1)}\bigl(X(I_{n-1}), \abs{I_{n-1}}\bigr) Q_{k(n)}\bigl(X(I_n), \abs{I_n}\bigr) \Bigr]
= 0.
\]
\end{Lemma}

\begin{proof}
$Q_1(X(J), \abs{J}) \Omega = X(J) \Omega = \chf{J}$. Using Lemma~\ref{Lemma:OP}, by induction
\[
\begin{split}
& \Bigl[Q_{k+1}(X(J), \abs{J}) + \alpha \abs{J} Q_k(X(J), \abs{J}) + \abs{J} Q_{k-1}(X(J), \abs{J}) \Bigr] \Omega \\
&\quad = X(J) \ Q_k(X(J), \abs{J}) \Omega
= X(J) \chf{J}^{\otimes k}
= \chf{J}^{\otimes (k+1)} + \alpha \abs{J} \chf{J}^{\otimes k} + \abs{J} \chf{J}^{\otimes (k-1)} \\
&\quad = \chf{J}^{\otimes (k+1)} + \alpha \abs{J} Q_k(X(J), \abs{J}) \Omega + \abs{J} Q_{k-1}(X(J), \abs{J}) \Omega.
\end{split}
\]
So the result holds for $n=1$. Similarly, again using Lemma~\ref{Lemma:OP},
\[
\begin{split}
& X(I_1) P_{k(1)}\bigl(X(I_1), \abs{I_1}\bigr) \chf{I_2}^{\otimes k(2)} \otimes \ldots \otimes \chf{I_n}^{\otimes k(n)}
= X(I_1) \chf{I_1}^{\otimes k(1)} \otimes \ldots \otimes \chf{I_n}^{\otimes k(n)} \\
&\quad = \Bigl[P_{k(1)+1}\bigl(X(I_1), \abs{I_1}\bigr) + \alpha \abs{I_1} P_{k(1)}\bigl(X(I_1), \abs{I_1}\bigr) + \abs{I_1} P_{k(1)-1}\bigl(X(I_1), \abs{I_1}\bigr) \Bigr] \\
&\qquad \chf{I_2}^{\otimes k(2)} \otimes \ldots \otimes \chf{I_n}^{\otimes k(n)}.
\end{split}
\]
Finally, for $I \cap I_1 = \emptyset$,
\[
\begin{split}
P_1(X(I), \abs{I}) \chf{I_1}^{\otimes k(1)} \otimes \ldots \otimes \chf{I_n}^{\otimes k(n)}
& = (X(I) - \alpha \abs{I}) \chf{I_1}^{\otimes k(1)} \otimes \ldots \otimes \chf{I_n}^{\otimes k(n)} \\
& = \chf{I} \otimes \chf{I_1}^{\otimes k(1)} \otimes \ldots \otimes \chf{I_n}^{\otimes k(n)}. \qedhere
\end{split}
\]
\end{proof}

\begin{Prop}
With respect to $(\phi, \psi_T)$, the process $\set{X(t): 0 \leq t \leq T}$ is an algebraic two-state free Brownian motion.
\end{Prop}

\begin{proof}
First we note that from the representation \eqref{Fock-rep}, the Jacobi parameters of the distribution of $X(t)$ with respect to $\phi$ are exactly those in equation~\eqref{Jacobi-mu}, so this distribution is $\mu_t$ from equation~\eqref{mu}. Also, for $t \leq T$,
\[
\begin{split}
\psi_T[X(t)^n]
& = \ip{\Omega}{X(t)^n (1 + \alpha X(T)) \Omega}
= \ip{\Omega}{X(t)^n (1 + \alpha X(t)) \Omega} \\
& = \phi[X(t)^n (1 + \alpha X(t)) ]
= \int_{\mf{R}} x^n (1 + \alpha x) \,d\mu_t(x)
= \int_{\mf{R}} x^n \,d\nu_t(x),
\end{split}
\]
where we have used Lemma~\ref{Lemma:OP}. Thus the distribution of $X(t)$ with respect to $\psi_T$ is $\nu_t$.

Now let $I_1, I_2, \ldots, I_n \subset [0,T]$ be intervals such that all $I_i \cap I_{i+1} = \emptyset$. Then
\[
\begin{split}
& \psi_T \Bigl[ P_{k(1)}\bigl(X(I_1), \abs{I_1}\bigr) \ldots P_{k(n-1)}\bigl(X(I_{n-1}), \abs{I_{n-1}}\bigr) P_{k(n)}\bigl(X(I_n), \abs{I_n}\bigr) \Bigr] \\
&\quad = \phi \Bigl[ P_{k(1)}\bigl(X(I_1), \abs{I_1}\bigr) \ldots P_{k(n-1)}\bigl(X(I_{n-1}), \abs{I_{n-1}}\bigr) P_{k(n)}\bigl(X(I_n), \abs{I_n}\bigr) (1 + \alpha X(I_n)) \Bigr] \\
&\quad = \phi \Bigl[ P_{k(1)}\bigl(X(I_1), \abs{I_1}\bigr) \ldots P_{k(n-1)}\bigl(X(I_{n-1}), \abs{I_{n-1}}\bigr) \\
&\qquad\qquad \Bigl(\alpha Q_{k(n)+1}\bigl(X(I_n), \abs{I_n}\bigr) + Q_{k(n)}\bigl(X(I_n), \abs{I_n}\bigr) \Bigr) \Bigr] = 0.
\end{split}
\]
On the other hand, consider
\[
\phi \Bigl[ P_{k(1)}\bigl(X(I_1), \abs{I_1}\bigr) \ldots P_{k(n-1)}\bigl(X(I_{n-1}), \abs{I_{n-1}}\bigr) P_{k(n)}\bigl(X(I_n), \abs{I_n}\bigr) \Bigr].
\]
We can expand $P_{k(n)}(x,t) = \sum_{i=0}^{k(n)} a_i(t) Q_i(x,t)$. Note that
\[
a_0(t) = \int_{\mf{R}} P_{k(n)}(x,t) \,d\mu_t(x) = \phi[P_{k(n)}(X(t), t)]
\]
Then using Lemma~\ref{Lemma:Fock-product},
\[
\begin{split}
& \phi \Bigl[ P_{k(1)}\bigl(X(I_1), \abs{I_1}\bigr) \ldots P_{k(n-1)}\bigl(X(I_{n-1}), \abs{I_{n-1}}\bigr) P_{k(n)}\bigl(X(I_n), \abs{I_n}\bigr) \Bigr] \\
&\quad = \phi \Bigl[ P_{k(1)}\bigl(X(I_1), \abs{I_1}\bigr) \ldots P_{k(n-1)}\bigl(X(I_{n-1}), \abs{I_{n-1}}\bigr) a_0(\abs{I_n}) \Bigr] \\
&\quad = \phi \Bigl[ P_{k(1)}\bigl(X(I_1), \abs{I_1}\bigr) \ldots P_{k(n-1)}\bigl(X(I_{n-1}), \abs{I_{n-1}}\bigr) \Bigr] \phi \Bigl[ P_{k(n)}\bigl(X(I_n), \abs{I_n}\bigr) \Bigr].
\end{split}
\]
Proceeding in this way, we conclude by induction that
\[
\begin{split}
& \phi \Bigl[ P_{k(1)}\bigl(X(I_1), \abs{I_1}\bigr) \ldots P_{k(n-1)}\bigl(X(I_{n-1}), \abs{I_{n-1}}\bigr) P_{k(n)}\bigl(X(I_n), \abs{I_n}\bigr) \Bigr] \\
&\quad = \phi \Bigl[ P_{k(1)}\bigl(X(I_1), \abs{I_1}\bigr) \Bigr] \ldots \phi \Bigl[ P_{k(n-1)}\bigl(X(I_{n-1}), \abs{I_{n-1}}\bigr) \Bigr] \phi \Bigl[ P_{k(n)}\bigl(X(I_n), \abs{I_n}\bigr) \Bigr].
\end{split}
\]
It follows that the increments of the process are two-state freely independent, and therefore the joint distribution of the process $\set{X(t)}$ is indeed that of the two-state free Brownian motion.
\end{proof}

\begin{Prop}
From now on, denote $T = 1/\alpha^2$. Denote $\mc{A}_T = W^\ast(\set{X(t): 0 \leq t \leq T})$. Then on $\mc{A}_T$, both $\phi$ and $\psi_T$ are faithful normal states, and moreover $\psi_T$ is tracial.
\end{Prop}

\begin{proof}
It follows from the preceding proposition that on $\mc{A}_T$, $\psi_T$ is positive, and so a state. Denote by $\rho$ the GNS representation of $C^\ast(\set{X(t): 0 \leq t \leq T})$ with respect to $\psi_T$, by $\widetilde{\mc{A}}_T$ the von Neumann algebra generated by the process in this representation, and by $\widetilde{\psi}_T$ the corresponding state on it. Since in $(\widetilde{\mc{A}}_T, \widetilde{\psi}_T)$, the process $\set{X(t): 0 \leq t \leq T}$ is a free Brownian motion, and it generates the algebra, the state $\widetilde{\psi}_T$ is the free product state of one-dimensional distributions, and as such is a faithful normal tracial state. Therefore $\psi_T = \widetilde{\psi}_T \circ \rho$ is a normal tracial state on $\mc{A}_T$. Moreover, to show that $\psi_T$ is faithful on $\mc{A}_T$, it suffices to show that $\rho$ is injective on it.

For $T = 1/\alpha^2$, by Corollary~\ref{Corollary:Densities} zero is in the spectrum of $C_T$ but is not an eigenvalue. Then we can find self-adjoint $B_n \in C^\ast(C_T)$ such that
\[
B_n C_T \Omega \rightarrow \Omega.
\]

Now suppose that $\rho(B) = 0$, then for any $A, C \in \mc{A}_T$
\[
0 = \lim_{n \rightarrow \infty} \psi_T[A B C B_n]
= \lim_{n \rightarrow \infty} \phi[A B C B_n C_T]
= \lim_{n \rightarrow \infty} \ip{A^\ast \Omega}{B C B_n C_T \Omega}
= \ip{A^\ast \Omega}{B C \Omega},
\]
which implies that $B = 0$ since $\Omega$ is cyclic for $\mc{A}_T$ as represented on $\mc{F}(L^2([0,T], \,dx))$. So the representation $\rho$ is injective on $\mc{A}_T$. Moreover, if $\phi[A^\ast A] = 0$, then by the Cauchy-Schwartz inequality
\[
\psi_T[A^\ast A] = \phi[A^\ast A C_T] = 0.
\]
Therefore $\phi$ is faithful.
\end{proof}

\begin{Prop}
\label{Prop:Larger-than-T}
For $t > 1/\alpha^2$, $\phi$ is not a faithful state on $\mc{A}_t$.
\end{Prop}

\begin{proof}
For $t > 1/\alpha^2$, by Corollary~\ref{Corollary:Densities} the distribution of $C_T$ has an atom at zero, and so $C_T$ has a non-trivial kernel. Indeed, the vector
\[
\eta = \sum_{n=0}^\infty \left( - \frac{1}{\alpha t} \right)^n \chf{[0,t)}^{\otimes n}
\]
is in this kernel; note that the norm of this vector is
\[
\sqrt{\sum_{n=0}^\infty \left( \frac{1}{\alpha^2 t} \right)^n} < \infty.
\]

Suppose $\xi \in \ker C_t$. We can write $\xi = \sum_{n=0}^\infty \xi_n$ with $\xi_n \in \mc{H}^{\otimes n}$. Note that on $\bigoplus_{n=1}^\infty \mc{H}^{\otimes n}$, $C_t$ acts in the same way as $1 + \alpha S(t)$, where $S(t)$ has the semicircular distribution, which has no atoms. So $\xi$ is not in this subspace, and $\xi_0 \neq 0$, so without loss of generality, $\xi_0 = \Omega$. But then $(\eta - \xi)$ is also in the kernel, and $(\eta - \xi)_0 = 0$. It follows that $\xi = \eta$ and so $\ker C_t = \mf{C} \eta$. We conclude that $\mc{A}_t$, and in fact $C^\ast(X(t))$, contains a rank-one operator $P_\eta : \zeta  \mapsto \ip{\eta}{\zeta} \eta$.

The non-zero positive operator
\[
\left(X(s) + \frac{s}{\alpha t} \right) P_\eta \left(X(s) + \frac{s}{\alpha t} \right)
\]
is in $\mc{A}_t$. We compute, for $s < t$,
\[
\left(X(s) + \frac{s}{\alpha t} \right) \eta
= \chf{[0,s)} \otimes \eta + \alpha s (\eta - \Omega) \perp \Omega.
\]
Denote this vector by $\eta(s)$. Then
\[
\phi \left[\left(X(s) + \frac{s}{\alpha t} \right) P_\eta \left(X(s) + \frac{s}{\alpha t} \right) \right] = \abs{\ip{\Omega}{\eta(s)}}^2 = 0.
\]
Therefore $\phi$ is not faithful on $\mc{A}_t$.
\end{proof}

\begin{Cor}
A two-state free Brownian motion $\set{X(t)}$ with parameter $\alpha$ can be realized in a two-state non-commutative probability space $(\mc{A}, \phi, \psi)$ with faithful normal $\phi$ and normal $\psi$ for $t \in [0, 1/\alpha^2]$ but not for larger values of $t$.
\end{Cor}

\begin{proof}
Note first that the Fock space representation of $\mc{A}_t$ on $\mc{F}(L^2([0,t], \,dx))$ considered in this section is exactly the GNS representation of $(\mc{A}_t, \phi)$.

The assumption that $\set{X(t)}$ is an (algebraic) two-state free Brownian motion determines the values of $\phi$ on the (non-closed) algebra generated by $\set{X(t)}$, and therefore, via the GNS representation, on the von Neumann algebra $\mc{A}_t$. For the same reason, the values of $\psi$ on the non-closed algebra are determined, and since $\psi$ is normal, it extends uniquely to $\mc{A}_t$. We conclude that the Fock realization is the unique realization of the two-state free Brownian motion $\set{X(t)}$ with parameter $\alpha$. But then the result follows from the preceding proposition.
\end{proof}

\begin{Remark}
The two-state free central limit theorem was proved in \cite{BLS96}, where the ($\phi$-)dis\-tri\-bu\-tions of the limit objects were computed. These distributions do \emph{not} belong to our class, but rather to the family considered in \cite{Ricard-t-Gaussian} and \cite{Wysoczanski-infinite-t}. Therefore the processes corresponding to those distributions do not exist in the analytic sense. Note that in the first of the papers just cited, the time convolution parameter is $n$, while $t$ is a fixed parameter.

A related question concerns the $8$-parameter family of two-state free convolution semigroups constructed in \cite{Ans-Mlot-Semigroups}: which of these semigroups are distributions of a process with two-state freely independent increments, which can be realized in a von Neumann algebra with a faithful state?
\end{Remark}

\begin{Prop}
For $\alpha \neq 0$, there is no $\phi$-preserving conditional expectation from $\mc{A}_T$ to $\mc{A}_t$.
\end{Prop}

\begin{proof}
Suppose the desired conditional expectation $E: \mc{A}_T \rightarrow \mc{A}_t$ exists. It then satisfies
\[
E[B_1 A B_2] = B_1 E[A] B_2
\]
and
\[
\phi[E[A]] = \phi[A]
\]
for $A \in \mc{A}_T$, $B_1, B_2 \in \mc{A}_t$. We compute, for $t < s < T$ and $B \in \mc{A}_t$,
\[
\begin{split}
\phi[B^\ast E[X(s) X(t)]] = \phi[B^\ast X(s) X(t)]
& = \ip{X(s) X(t) \Omega}{B \Omega} \\
& = \ip{\chf{[0,s)} \otimes \chf{[0,t)} + t \Omega + \alpha s \chf{[0,t)}}{B \Omega} \\
& = \ip{\chf{[0,t)} \otimes \chf{[0,t)} + t \Omega + \alpha s \chf{[0,t)}}{B \Omega} \\
& = \ip{(X(t)^2 + \alpha (s-t) X(t)) \Omega}{B \Omega} \\
& = \phi[B^\ast (X(t)^2 + \alpha (s-t) X(t))].
\end{split}
\]
Since $\phi$ is faithful on $\mc{A}_t$, this implies that
\[
E[X(s) X(t)] = X(t)^2 + \alpha (s-t) X(t).
\]
On the other hand, by a similar argument $E[X(s)] = X(t)$ and so $E[X(s) X(t)] = X(t)^2$. We arrive at a contradiction.
\end{proof}

\begin{Remark}[Polynomial martingales]
\label{Remark:Martingale}
For $B \in \mc{A}_t$ and $s > t$,
\[
\phi[B^\ast Q_n(X(s), s)]
= \ip{\chf{[0,s)}^{\otimes n}}{B \Omega}
= \ip{\chf{[0,t)}^{\otimes n}}{B \Omega}
= \phi[B^\ast Q_n(X(t), t)].
\]
So if we did have conditional expectations, then the process would have polynomial martingales and the Markov property.
\end{Remark}

\begin{Remark}[Generator]
Even though the two-state free Brownian motion with parameter $\alpha$ is not (for $\alpha \neq 0$) a Markov process, as noted in Remark~\ref{Remark:Bi-Poisson}, it has a classical version which is a Markov process. Denote by $\mc{K}_{s,t}$ the transition functions of the classical version. The operator $A_t$ is the \emph{generator} of the process at time $t$ if for some dense domain $\mc{D} \subset L^2(\mf{R}, \,d\mu_t)$ and any $f \in \mc{D}$,
\[
\left. \frac{\partial}{\partial h}\right|_{h=0} \mc{K}_{t,t+h}(f) = A_t f.
\]
See \cite{Ans-Generator} for related ideas.
\end{Remark}

\begin{Prop}
On the dense domain of polynomial functions, the generator of the two-state free Brownian motion with parameter $\alpha$ is
\[
\alpha (\partial_x - L_{\mu_t}) + \partial_x L_{\nu_t},
\]
where we use the notation \cite{Ans-Bochner}
\[
L_\nu[f](x) = \int_{\mf{R}} \frac{f(x) - f(y)}{x - y} \,d\nu(y) = (1 \otimes \nu) (\partial f),
\]
and $\partial$ is the difference quotient.
\end{Prop}

\begin{proof}
Using Remark~\ref{Remark:Martingale} (or Proposition~3.3 of \cite{Bryc-Meixner} and Lemma~2.4 of \cite{Bryc-Wesolowski-Bi-Poisson}), the orthogonal polynomials $Q_n(x,t)$ are martingale polynomials for the (classical version of) the process. It is easy to see that, to show that $A_t$ is the generator of the process with the domain consisting of all polynomial functions, it suffices to show that
\[
\partial_t Q_n(x,t) = - A_t Q_n(x,t)
\]
for all $n$. Note that since $\mu_t = \Phi_t[\nu_t]$, by Lemma~7 of \cite{AnsAppell3} the polynomials $Q_n(x,t)$ are precisely the c-free Appell polynomials for this pair. By Proposition~4 from the same paper, the generating function for these polynomials is
\[
H(x,t,z) = \sum_{n=0}^\infty Q_n(x,t) z^n = \frac{1 + t \alpha z}{1 - x z + t(\alpha z + z^2)}
\]
since $R^{\nu_t}(z) = t (\alpha z + z^2)$ and $R^{\nu_t}(z) - R^{\mu_t, \nu_t}(z) = t \alpha z$; this result is also easy to obtain directly. On the other hand,
\[
\sum_{n=0}^\infty P_n(x,t) z^n = \frac{1}{1 - x z + t(\alpha z + z^2)}.
\]
We conclude that
\[
\partial H(x,t,z) = z H(x,t,z) \frac{1}{1 - y z + t(\alpha z + z^2)}
\]
and so
\[
L_{\nu_t}[H(x,t,z)] = z H(x,t,z)
\]
and
\[
L_{\mu_t}[H(x,t,z)] = \frac{z}{1 - x z + t (\alpha z + z^2)}.
\]

Now we compute
\[
\begin{split}
& \Bigl(\alpha (\partial_x - L_{\mu_t}) + \partial_x L_{\nu_t} \Bigr) [H(x,t,z)] \\
&\quad = \alpha z \frac{1 + t \alpha z}{(1 - x z + t(\alpha z + z^2))^2} - \alpha z \frac{1}{1 - x z + t (\alpha z + z^2)} + z^2 \frac{1 + t \alpha z}{(1 - x z + t(\alpha z + z^2))^2} \\
&\quad = - \partial_t \frac{1 + t \alpha z}{1 - x z + t(\alpha z + z^2)}
= - \partial_t H(x,t,z).
\end{split}
\]
The result follows.
\end{proof}

\begin{Remark}[It{\^o} formula]
By the same methods as in \cite{BiaSpeBrownian} and \cite{AnsIto}, for sufficiently nice $f$,
\begin{equation}
\label{Ito}
f(X(t)) = f(X(0)) + \int_0^t \partial f(X(s)) \sharp \,dX(s) + \int_0^t (\partial_x \otimes \psi) \partial f(X(s)) \,ds,
\end{equation}
where we use the notation
\[
\int_0^t (A(x) \otimes B(s)) \sharp \,dX(s) = \int _0^t A(x) \,dX(s) \, B(s).
\]
Using Lemma~2.1 of \cite{BLS96} and the observation that the process $\set{X(t)}$ is $\phi$-centered, we see that
\[
\phi \left[ \int_0^t \partial f(X(s)) \sharp \,dX(s) \right]
= \int_0^t \Bigl( \phi \left[ \partial_x f(X(s)) \right] - (\phi \otimes \phi) \left[ (\partial f)(X(s)) \right] \Bigr) d\psi[X(s)].
\]
Therefore taking $\phi$-expectations in the It{\^o} formula~\eqref{Ito} gives
\[
\begin{split}
\phi[f(X(t))]
& = \phi[f(X(0))] + \int_0^t \Bigl( \phi \left[ \partial_x f(X(s)) \right] - (\phi \otimes \phi) \left[ (\partial f)(X(s)) \right] \Bigr) d\psi[X(s)] \\
&\quad + \int_0^t \phi\left[ (\partial_x \otimes \psi) \partial f(X(s))\right] \,ds \\
& = \phi[f(X(0))] + \int_0^t \phi \Bigl[ \Bigl( \alpha \partial_x - \alpha (1 \otimes \phi) \partial + (\partial_x \otimes \psi) \partial \Bigr) f(X(s)) \Bigr] \,ds.
\end{split}
\]
This result is consistent with the generator formula in the preceding proposition.
\end{Remark}

\section{$C^\ast$-algebra setting}
\label{Section:Cstar}

We saw in Corollary~\ref{Cor:Variations} that for any algebraic two-state free Brownian motion, $\int_0^T (dX(t))^k = 0$ for $k > 2$ and $\int_0^T (dX(t))^2 = T$ as limits in $L^2(\phi)$. If $\phi$ is a faithful state, these limits can be identified with elements in $\mc{A}$. We now investigate the same limits in $L^\infty(\phi)$. Here
\[
\norm{A}_\infty = \lim_{n \rightarrow \infty} \norm{A}_{2n}
\]
and for $A \in \mc{A}^{sa}$,
\[
\norm{A}_{2n} = \phi[A^{2n}]^{1/2n}.
\]
Note that if $\phi$ is faithful, then $\norm{A}_\infty = \norm{A}$, the operator norm on $\mc{A}$.

\begin{Lemma}
\label{Lemma:Stirling}
Recall that the Stirling number of the second kind $S(n,k)$ is the number of set partitions of a set of $n$ elements into $k$ non-empty blocks. Then
\[
\lim_{n \rightarrow \infty} \left( \sum_{k=1}^N S(n,k) \right)^{1/n} = N.
\]
\end{Lemma}

\begin{proof}
It is easy to see that
\[
\frac{n!}{N! ((n/N)!)^{N}} \leq S(n, N) \leq \sum_{k=1}^N S(n,k) \leq N^n.
\]
The result now follows by Stirling's formula.
\end{proof}

\begin{Thm}
Suppose $\set{X(t)}$ is an algebraic two-state free Brownian motion.
\begin{enumerate}
\item
Assume further that all the $\psi$-free cumulants of $X$ are non-negative. Then $\int_0^T (dX(t))^p = 0$ as a limit in $L^\infty(\phi)$ for some $p > 2$ if and only if $R^\psi_k(X) = 0$ for all $k > 2$. In this case, in fact $\int_0^T (dX(t))^p = 0$ for all $p > 2$.
\item
Assume now that $R^\psi_k(X) = 0$ for $k > 2$. Then $\int_0^T (dX(t))^2 = T$ as a limit in $L^\infty(\phi)$ if and only if $R^\psi_2(X) = T$, so that $\set{X(t)}$ is a two-state free Brownian motion with parameter $\alpha$.
\end{enumerate}
\end{Thm}

\begin{proof}
For part (a), using both assumptions on the process,
\[
\phi \left[ \left(\sum_{i=1}^N X_{i,N}^p \right)^n \right]
\geq \phi \left[ \sum_{i=1}^N X_{i,N}^{p n} \right]
\geq N \left( \frac{1}{N} R^{\phi, \psi}_2(X) \right) \left( \frac{1}{N} R^\psi_{p n - 2}(X) \right)
= \frac{1}{N} T R^\psi_{p n - 2}(X).
\]
Therefore
\[
\norm{\sum_{i=1}^N X_{i,N}^p}_{2n}
\geq \left( \frac{1}{N} T R^\psi_{2 p n - 2}(X) \right)^{1/2n}
\]
and
\[
\norm{\sum_{i=1}^N X_{i,N}^p}_\infty
\geq \limsup_{n \rightarrow \infty} R^\psi_{2 p n - 2}(X)^{1/2n}.
\]
So to have $\lim_{N \rightarrow \infty} \norm{\sum_{i=1}^N X_{i,N}^p}_\infty = 0$, we need $\limsup_{n \rightarrow \infty} R^\psi_{2 p n-2}(X)^{1/2n} = 0$. Denote by $R^{\nu_t}(z)$ the generating function for the $\psi$-free cumulants of $X(t)$. Since $\set{\nu_t}$ form a free convolution semigroup, and all their moments are finite, we have the free canonical representation (Theorem~6.2 of \cite{Maa92})
\[
R^{\nu_t}(z) = t \left( \alpha z + \int_{\mf{R}} \frac{z^2}{1 - x z} \,d\lambda(x) \right)
\]
for a finite positive measure $\lambda$ (our $R$ is $z$ times the usual $R$-transform). In particular, for $n \geq 2$, $R^\psi_{n+2}(X(t)) = t \int_{\mf{R}} x^n \,d\lambda(x)$. Since by H{\"o}lder's inequality,
\[
\int_{\mf{R}} \abs{x}^{n-1} \,d\lambda(x) \leq \left( \int_{\mf{R}} \abs{x}^{n} \,d\lambda(x) \right)^{(n-1)/n} \lambda(\mf{R})^{1/n},
\]
in fact $\limsup_{n \rightarrow \infty} \abs{R^\psi_{n}(X)}^{1/n} = 0$. This says that $R^{\nu_t}(z)$ is analytic in the complex plane. It follows that the Cauchy transform
\[
G_\lambda(z) = \int_{\mf{R}} \frac{1}{z - x} \,d\lambda(x) = \frac{z}{t} R^{\nu_t}(1/z) - \alpha
\]
is also analytic, except possibly at $z=0$. But then by the Stieltjes inversion formula
\[
d\lambda(x) = - \frac{1}{\pi} \lim_{y \downarrow 0} \ \mathrm{Im} \ G_\lambda(x + i y),
\]
$\lambda$ is a multiple of $\delta_0$. So $R^{\nu_t}(z) = t (\alpha z + \beta z^2)$, and $\nu_t$ is a semicircular distribution.

On the other hand, if $R^\psi_1(X) = \alpha T$, $R^\psi_2(X) = \beta T$, and $R^\psi_k(X) = 0$ for $k > 2$, then
\[
\phi \left[ \left(\sum_{i=1}^N X_{i,N}^k \right)^n \right]
\leq N^n \frac{1}{N^{k n/2}} \abs{\NC(kn)} \max(1, T, \abs{\alpha} T, \beta T)^{k n},
\]
\[
\norm{\sum_{i=1}^N X_{i,N}^k}_\infty \leq \frac{1}{N^{k/2-1}} 4^k \max(1, T, \abs{\alpha} T, \beta T)^{k},
\]
and
\[
\lim_{N \rightarrow \infty} \norm{\sum_{i=1}^N X_{i,N}^k}_\infty = 0.
\]

For part (b), we first assume that $\beta \neq 1$ and use

\begin{Lemma}
\label{Lemma:Sum}
\[
\phi \left[ \left(\sum_{i=1}^N X_{i,N}^2 - T \right)^n \right] = \sum_{\substack{\sigma \in \NC_{1,2}(2n) \\ \Outer(\sigma) \cap \Sing(\sigma) = \emptyset \\ \Outer(\sigma) \cap \tau_n = \emptyset}} \left( \sum_{\substack{\pi \in \Part(2n) \\ \pi \geq (\sigma \vee \tau_n)}} N_{\abs{\pi}} \right) \left(\frac{T}{N}\right)^{\abs{\sigma}} \alpha^{\abs{\Sing(\sigma)}}(\beta - 1)^{\abs{\Inner(\sigma) \cap \tau_n}},
\]
where $N_n = N (N-1) \ldots (N-n+1)$ and
\[
\tau_n = \set{(1,2), (3,4), \ldots, (2n-1, 2n)}.
\]
\end{Lemma}

Note also that $N_{\abs{\pi}} = 0$ for $\abs{\pi} > N$.

Now take
\[
\sigma = \set{(1, 2n), (2, 2n-1), (3,4), (5,6), \ldots, (2n-3, 2n-2)}.
\]
Then
\[
\sigma \vee \tau_n = \set{(1, 2, 2n-1, 2n), (3,4), (5,6), \ldots, (2n-3, 2n-2)}
\]
and $\abs{\sigma \vee \tau_n} = n-1$. Therefore
\[
\abs{\set{\pi \in \Part(2n): \pi \geq (\sigma \vee \tau_n), \abs{\pi} \leq N}} = \abs{\set{\pi \in \Part(n-1), \abs{\pi} \leq N}} = \sum_{k=1}^{N} S(n-1, k).
\]
So using only the term corresponding to $\sigma$,
\[
\phi \left[ \left(\sum_{i=1}^N X_{i,N}^2 - T \right)^n \right]
\geq \frac{1}{N^{n}} \sum_{k=1}^{N} S(n-1, k) T^n (\beta - 1)^{n-2}
\]
and by Lemma~\ref{Lemma:Stirling},
\[
\norm{\sum_{i=1}^N X_{i,N}^2 - T}_\infty \geq T \abs{\beta - 1}
\]
for all $N$.

On the other hand, if $\beta = 1$, then
\[
\phi \left[ \left(\sum_{i=1}^N X_{i,N}^2 - T \right)^n \right]
= \sum_{\substack{\sigma \in \NC_{1,2}(2n) \\ \Outer(\sigma) \cap \Sing(\sigma) = \emptyset \\ \sigma \wedge \tau_n = \hat{0}}}
\left( \sum_{\substack{\pi \in \Part(2n) \\ \pi \geq (\sigma \vee \tau_n)}} N_{\abs{\pi}} \right) \left(\frac{T}{N}\right)^{\abs{\sigma}} \alpha^{\abs{\Sing(\sigma)}}.
\]
In this case, the conditions $\sigma \in \NC_{1,2}(2n)$, $\sigma \wedge \tau_n = \hat{0}$ guarantee that $\abs{\sigma \vee \tau_n} \leq \abs{\sigma} - \frac{n}{2}$, and so for each such $\sigma$,
\[
\abs{\set{\pi \in \Part(2n): \pi \geq (\sigma \vee \tau_n), \abs{\pi} \leq N}}
= \abs{\set{\pi \in \Part(\abs{\sigma \vee \tau_n}), \abs{\pi} \leq N}}
\leq N^{\abs{\sigma \vee \tau_n}}
\leq N^{\abs{\sigma} - n/2}.
\]
This time we also note that $N_{\abs{\pi}} \leq N^N$. Then
\[
\begin{split}
\phi \left[ \left(\sum_{i=1}^N X_{i,N}^2 - T \right)^n \right]
& \leq \sum_{\substack{\sigma \in \NC_{1,2}(2n) \\ \sigma \wedge \tau_n = \hat{0}}}
N^{\abs{\sigma \vee \tau_n} - \abs{\sigma}} N^N \max(1, T, \abs{\alpha} T)^{2n} \\
& \leq 4^{2n} N^{-n/2} N^N \max(1, T, \abs{\alpha} T)^{2n}
\end{split}
\]
and
\[
\norm{\sum_{i=1}^N X_{i,N}^2 - T}_{\infty} \leq \frac{1}{\sqrt{N}} 4^2 \max(1, T, \abs{\alpha} T)^2,
\]
so
\[
\lim_{N \rightarrow \infty} \norm{\sum_{i=1}^N X_{i,N}^2 - T}_\infty = 0. \qedhere
\]
\end{proof}

\begin{proof}[Proof of Lemma~\ref{Lemma:Sum}]
We first note that
\[
\binom{n}{k} \phi \left[ \left(\sum_{i=1}^N X_{i,N}^2 \right)^{n-k} T^k \right]
= \sum_{i(1), i(2), \ldots, i(n) = 1}^N \ \sum_{\substack{S \subset \set{1, 2, \ldots, n}, \\ \abs{S} = k}} \phi\left[Y_{i(1), 1, N}^2 Y_{i(2), 2, N}^2\ldots Y_{i(n), n, N}^2 \right],
\]
where
\[
Y_{i, j, N} =
\begin{cases}
X_{i, N}, & j \not \in S, \\
\sqrt{\frac{T}{N}}, & j \in S.
\end{cases}
\]
This expression equals
\[
\begin{split}
& \sum_{i(1), i(2), \ldots, i(n) = 1}^N \ \sum_{\substack{S \subset \set{1, 2, \ldots, n}, \\ \abs{S} = k}} \ \sum_{\sigma \in \NC(2n)} \prod_{V \in \Outer(\sigma)} R^{\phi, \psi}(Y_{i(1), 1, N}, Y_{i(1), 1, N}, \ldots, Y_{i(n), n, N} : V) \\
&\qquad \prod_{V \in \Inner(\sigma)} R^{\psi}(Y_{i(1), 1, N}, Y_{i(1), 1, N}, \ldots, Y_{i(n), n, N} : V).
\end{split}
\]
Since $R^{\phi, \psi}_m(X_{i,N}) = 0$ for $m \neq 2$ and $R^\psi_m(X_{i,N}) = 0$ for $m > 2$, this simplifies to
\[
\begin{split}
& \sum_{i(1), i(2), \ldots, i(n) = 1}^N \ \sum_{\substack{U \subset \tau_n \\ \abs{U} = k}} \ \sum_{\substack{\sigma \in \NC_{1,2}(2n) \\ U \subset \sigma \\ \Outer(\sigma) \cap \Sing(\sigma) = \emptyset}} \prod_{(u, v) \in U} \frac{T}{N} \prod_{(u, v) \in \Outer(\sigma) \setminus U} R^{\phi, \psi} \left(X_{i([(u+1)/2]), N}, X_{i([(v+1)/2]), N} \right) \\
&\qquad \prod_{(u, v) \in \Inner(\sigma) \setminus U} R^{\psi} \left(X_{i([(u+1)/2]), N}, X_{i([(v+1)/2]), N} \right) \prod_{(u) \in \Inner(\sigma)} R^{\psi} \left(X_{i([(u+1)/2]), N} \right),
\end{split}
\]
where $[a]$ denotes the integer part.

For each choice of $(i(1), i(2), \ldots, i(n))$, we define the partition $\pi \in \Part(2n)$ by $(2j-1) \stackrel{\pi}{\sim} (2j)$ and
\[
2j_1 \stackrel{\pi}{\sim} 2 j_2 \Leftrightarrow i(j_1) = i(j_2).
\]
Then the sum is transformed into
\[
\sum_{\substack{U \subset \tau_n \\ \abs{U} = k}} \ \sum_{\substack{\sigma \in \NC_{1,2}(2n) \\ U \subset \sigma \\ \Outer(\sigma) \cap \Sing(\sigma) = \emptyset}} \sum_{\substack{\pi \in \Part(2n) \\ \pi \geq \sigma \\ \pi \geq \tau_n}} N_{\abs{\pi}} \prod_{(u, v) \in U} \frac{T}{N} \prod_{(u,v) \in \Outer(\sigma) \setminus U} \frac{T}{N} \prod_{(u) \in \sigma} \alpha \frac{T}{N} \prod_{(u,v) \in \Inner(\sigma) \setminus U} \beta \frac{T}{N}.
\]
Therefore
\[
\begin{split}
& \phi \left[ \left(\sum_{i=1}^N X_{i,N}^2 - T \right)^n \right]
= \sum_{k=0}^n (-1)^k \binom{n}{k} \phi \left[ \left(\sum_{i=1}^N X_{i,N}^2 \right)^{n-k} T^k \right] \\
& = \sum_{k=0}^n (-1)^k \sum_{\substack{\sigma \in \NC_{1,2}(2n) \\ \Outer(\sigma) \cap \Sing(\sigma) = \emptyset}} \
\sum_{\substack{U \subset (\sigma \cap \tau_n) \\ \abs{U} = k}}
\left( \sum_{\substack{\pi \in \Part(2n) \\ \pi \geq (\sigma \vee \tau_n)}} N_{\abs{\pi}} \right) \left(\frac{T}{N}\right)^{\abs{\sigma}} \\
&\qquad \prod_{(u, v) \in U} 1 \prod_{(u,v) \in \Outer(\sigma) \setminus U} 1 \prod_{(u) \in \sigma} \alpha \prod_{(u,v) \in \Inner(\sigma) \setminus U} \beta \\
& = \sum_{\substack{\sigma \in \NC_{1,2}(2n) \\ \Outer(\sigma) \cap \Sing(\sigma) = \emptyset}}  \left( \sum_{\substack{\pi \in \Part(2n) \\ \pi \geq (\sigma \vee \tau_n)}} N_{\abs{\pi}} \right) \left(\frac{T}{N}\right)^{\abs{\sigma}} \alpha^{\abs{\Sing(\sigma)}} (1 - 1)^{\abs{\Outer(\sigma) \cap \tau_n}} (\beta - 1)^{\abs{\Inner(\sigma) \cap \tau_n}} \\
& = \sum_{\substack{\sigma \in \NC_{1,2}(2n) \\ \Outer(\sigma) \cap \Sing(\sigma) = \emptyset \\ \Outer(\sigma) \cap \tau_n = \emptyset}} \left( \sum_{\substack{\pi \in \Part(2n) \\ \pi \geq (\sigma \vee \tau_n)}} N_{\abs{\pi}} \right) \left(\frac{T}{N}\right)^{\abs{\sigma}} \alpha^{\abs{\Sing(\sigma)}}(\beta - 1)^{\abs{\Inner(\sigma) \cap \tau_n}}. \qedhere
\end{split}
\]
\end{proof}

\begin{Cor}
Let $(\mc{A}, \phi, \psi)$ be a $C^\ast$-noncommutative probability space, so that $\mc{A}$ is a $C^\ast$-algebra, $\phi$ and $\psi$ states on it, and $\phi$ is faithful. Suppose that $\mc{A}$ is generated by an algebraic two-state free Brownian motion $\set{X(t)}$ all of whose $\psi$-free cumulants are non-negative. Suppose also that $\int_0^T (dX(t))^k = 0$ for $k > 2$ and $\int_0^T (dX(t))^2 = T$, where the limits are taken in the operator norm. Then $\set{X(t)}$ is a two-state free Brownian motion with parameter $\alpha$.
\end{Cor}

\begin{Remark}
In the setting of the preceding corollary, our results do not imply directly that if $\mc{A}$ is generated by an algebraic two-state free Brownian motion $\set{X(t)}$ without any extra assumptions, then $\set{X(t)}$ has to be a two-state free Brownian motion with parameter $\alpha$. So it is possible that new examples may arise if we only assume that the state $\phi$ is faithful on the $C^\ast$-algebra and not on the von Neumann algebra generated by the process.

On the other hand, note that the argument in Proposition~\ref{Prop:Larger-than-T} shows that for a two-state free Brownian motion with parameter $\alpha$, for $t > 1/\alpha^2$ the state $\phi$ is not faithful even on the $C^\ast$-algebra.
\end{Remark}


\def\cprime{$'$}
\providecommand{\bysame}{\leavevmode\hbox to3em{\hrulefill}\thinspace}
\providecommand{\MR}{\relax\ifhmode\unskip\space\fi MR }
\providecommand{\MRhref}[2]{%
  \href{http://www.ams.org/mathscinet-getitem?mr=#1}{#2}
}
\providecommand{\href}[2]{#2}

\end{document}